\documentclass[12pt,a4paper]{amsart}
\usepackage[nocompress]{cite}
\usepackage{listings}
\usepackage{docmute}
\usepackage{amssymb, enumerate, calligra}
\usepackage{amsthm}
\usepackage{amsmath}
\usepackage{amsxtra}
\usepackage{latexsym}
\usepackage{mathrsfs}
\usepackage[all,cmtip]{xy}
\usepackage[all]{xy}
\usepackage{enumitem}
\usepackage{xcolor}
\usepackage{graphicx}
\usepackage{comment}
\usepackage{mathabx,epsfig}
\def\acts{\ \rotatebox[origin=c]{-90}{$\circlearrowright$}\ }
\def\racts{\ \rotatebox[origin=c]{90}{$\circlearrowleft$}\ }
\def\dacts{\ \rotatebox[origin=c]{0}{$\circlearrowright$}\ }
\def\uacts{\ \rotatebox[origin=c]{180}{$\circlearrowright$}\ }

\usepackage[dvipdfmx]{hyperref}

\newtheorem{thm}{Theorem}[section]
\newtheorem{lem}[thm]{Lemma}
\newtheorem{conj}[thm]{Conjecture}
\newtheorem{claim}[thm]{Claim}
\newtheorem{prop}[thm]{Proposition}

\theoremstyle{definition}
\newtheorem{defn}[thm]{Definition}
\newtheorem{rmk}[thm]{Remark}
\newtheorem*{ack}{Acknowledgements}
\numberwithin{equation}{section}

\def\C{{\mathbb C}}
\def\Q{{\mathbb Q}}
\def\R{{\mathbb R}}
\def\Z{{\mathbb Z}}
\def\P{{\mathbb P}}

\def\QQ{\overline{\mathbb Q}}

\DeclareMathOperator{\Div}{Div}

\DeclareMathOperator{\Pic}{Pic}

\DeclareMathOperator{\id}{id}

\DeclareMathOperator{\End}{End}

\DeclareMathOperator{\Supp}{Supp}

\DeclareMathOperator{\Nef}{Nef}

\DeclareMathOperator{\Eff}{Eff}
\DeclareMathOperator{\ord}{ord}

\newenvironment{claimproof}[0]
  {%
   \paragraph{\it Proof.}%
  }
  {%
    \hfill$\blacksquare$%
  }

  {\end{list}}

\renewcommand{\labelenumi}{\rm(\arabic{enumi})}

\title[]
{Kawaguchi-Silverman conjecture for endomorphisms on rationally connected varieties admitting an int-amplified endomorphism}
\author{Yohsuke Matsuzawa}
\author{Shou Yoshikawa}
\address{Graduate school of Mathematical Sciences, the University of Tokyo, Komaba, Tokyo,
153-8914, Japan}
\address{Graduate school of Mathematical Sciences, the University of Tokyo, Komaba, Tokyo,
153-8914, Japan}
\email{myohsuke@ms.u-tokyo.ac.jp}
\email{yoshikaw@ms.u-tokyo.ac.jp}

\begin{document}

\begin{abstract}
We prove Kawaguchi-Silverman conjecture for all surjective endomorphisms on every smooth rationally connected variety
admitting an int-amplified endomorphism.
\end{abstract}

\maketitle

\setcounter{tocdepth}{1}
\tableofcontents

\section{Introduction}

Let $X$ be a smooth projective variety and $f \colon X \dashrightarrow X$ a dominant rational self-map,
both defined over $ \overline{\mathbb Q}$.
Silverman introduced the notion of arithmetic degree in \cite{sil},
which measures the arithmetic complexity of $f$-orbits by means of Weil height functions.

On  the other hand, we can attach the dynamical degree $\delta_{f}$ to $f$, which  measures the geometric complexity of
the dynamical system.
In \cite{sil}, \cite[Conjecture 6]{ks3} Kawaguchi and Silverman conjectured that the arithmetic degree of any Zariski dense orbits
are equal to the first dynamical degree $\delta_{f}$ (cf.  Conjecture \ref{ksc}). 

If the Kodaira dimension of $X$ is positive, $f$ does not have Zariski dense orbits.
Therefore, the conjecture is interesting only for varieties of non-positive Kodaria dimension.
Smooth projective varieties of Kodaira dimension zero are isomorphic up to \'etale covering to products of abelian varieties, hyper-K\"ahler varieties, and Calabi-Yau varieties.
The conjecture is solved for abelian varieties \cite{ks1, sil2}, endomorphisms on hyper-K\"ahler varieties \cite{ls}.  
As for endomorphisms on varieties of Kodaira dimension zero, the essential remaining case is automorphisms on Calabi-Yau varieties.
On the other hand, smooth projective varieties of Kodaira dimension $-\infty$ are fibered (by rational maps) over non-uniruled varieties
(so non-negative Kodaira dimension conjecturely) with general fiber rationally connected. 
Thus, it is natural to tackle the conjecture for rationally connected varieties as a first step. 

One strategy to prove the conjecture for higher dimensional algebraic varieties is to use minimal model program (MMP) and reduce the problem 
to a problem on relatively easier varieties.
As for non-invertible endomorphisms on smooth projective surfaces, the conjecture is proved by using this strategy in \cite{mss}.
Such strategy also works for endomorphisms on higher dimensional projective varieties
if we further assume that the varieties admit int-amplified endomorphisms.
(A self-morphism $f \colon X \longrightarrow X$ of a normal projective variety $X$ is called int-amplified if there exists an ample divisor $H$ on $X$ such
that $f^{*}H-H$ is ample, cf. Definition \ref{def:intamp}. 
For varieties admitting int-amplified endomorphisms, a theory of MMP equivariant under endomorphisms is developed by Meng and Zhang, cf. \cite{meng-zhang, meng-zhang2, meng}.)
The main theorem of this paper is the following.
\begin{thm}\label{kscRC}
Let $X$ be a smooth projective rationally connected variety over $\QQ$ admitting an int-amplified endomorphism.
Then Kawaguchi-Silverman conjecture holds for all surjective endomorphisms on $X$, i.e.
the arithmetic degree $ \alpha_{f}(x)$ is equal to the dynamical degree $ \delta_{f}$ for $x\in X(\QQ)$ with Zariski dense $f$-orbit.
\end{thm}

\begin{rmk}
Meng and Zhang prove this theorem for threefolds in \cite{mz}.
\end{rmk}

The organization of this paper is as follows:
In \S \ref{sec:def}, we give definitions of dynamical and arithmetic degrees and summarize basic properties of them without proof.
We also introduce Kawaguchi-Silverman conjecture in this section.

The proof of the main theorem involves several techniques from birational geometry.
We summarize the statement of equivariant MMP in \S \ref{sec:eqmmp}.
A main difficulty of the proof is that certain kind of fibrations make it impossible to conclude
Kawaguchi-Silverman conjecture from that of the output of MMP.
We actually prove that such fibrations do not appear in our setting.
\S \ref{sec:covthm} is for the purpose of this.
In \S \ref{sec:prf}, we give the proof of the main theorem.

\vspace{15pt}
{\bf Notation and Terminology.}
\begin{itemize}
\item Throughout this paper, the ground field is $\QQ$ unless otherwise stated.
All statements that are purely geometric hold over any algebraically closed field of characteristic zero.
\item A variety over a filed $k$ is a geometrically integral separated scheme of finite type over $k$.
A divisor on a variety means a Cartier divisor.
A $\Q$-Cartier divisor (resp. $\R$-Cartier divisor) on a variety $X$ is an element of $(\Div X) {\otimes}_{\Z}\Q$ (resp. $(\Div X) {\otimes}_{\Z}\R$)
where $\Div X$ is the group of Cartier divisors.
When $X$ is normal, these groups are embedded in $Z^{1}(X) {\otimes}_{\Z}\Q$ (resp. $Z^{1}(X) {\otimes}_{\Z}\R$)
where $Z^{1}(X)$ is the group of codimension one cycles on $X$.
An element of $Z^{1}(X) {\otimes}_{\Z}\Q$ (resp.  $Z^{1}(X) {\otimes}_{\Z}\R$) is called $\Q$-Weil divisor (resp. $\R$-Weil divisor).
Linear equivalence and $\Q$-linear equivalence are denoted by $\sim$ and $\sim_{\Q}$ respectively. For $\Q$-Weil divisors $D$ and $E$, $D \sim E$ means $D-E$ is a principal divisor.
\item Let $X$ be a projective variety over an algebraically closed field of characteristic zero.
\begin{itemize}
\item $N^{1}(X)$ is the group of Cartier divisors modulo numerical equivalence
(a Cartier divisor $D$ is numerically equivalent to zero, which is denoted by $D\equiv0$, if $(D\cdot C)=0$ for all irreducible curves $C$ on $X$).
\item $N_{1}(X)$ is the group of $1$-cycles modulo numerical equivalence
(a $1$-cycle $ \alpha$ is numerically zero if $(D\cdot \alpha)=0$ for all Cartier divisors $D$).
By definition, $N^{1}(X)$ and $N_{1}(X)$ are dual to each other.
\item When $X$ is normal, the Iitaka dimension of a $\Q$-Cartier divisor $D$ on $X$ is denoted by $\kappa(D)$. 
\end{itemize}
\item A morphism $f \colon X \longrightarrow X$ from a projective variety $X$ to itself is called self-morphism of $X$ or endomorphism on $X$.
If it is surjective, then it is a finite morphism.
\item A morphism $f \colon X \longrightarrow Y$  between varieties is called an algebraic fiber space if $f$ is proper surjective and $f_{*} \mathcal{O}_{X}= \mathcal{O}_{Y}$.
\item A morphism $f \colon X \longrightarrow Y$  between varieties is called quasi-\'etale if $f$ is \'etale at every codimension one point of $X$.
\item Let $f \colon X \longrightarrow Y$ be a finite surjective morphism between normal varieties. The ramification divisor of $f$ is denoted by $R_{f}$.
\item The Picard number of a projective variety $X$ is denoted by $\rho(X)$.
\item The function field of a variety $X$ is denoted by $k(X)$.
\item Let $f \colon X \longrightarrow X$ be a self-morphism of a variety $X$.  A subset $S\subset X$ is called totally invariant under $f$ if $f^{-1}(S)=S$ as sets.
\item Consider the commutative diagram of the following form:
\[
\xymatrix{
X \ar@{-->}[r]^{\pi} \ar[d]_{f} & Y \ar[d]^{g} \\
X \ar@{-->}[r]_{\pi} & Y
}
\]
where $f, g$ are surjective morphisms and $\pi$ is a dominant rational map.
We write this diagram as 
\[
\xymatrix{
f \acts X \ar@{-->}[r]^{\pi}& Y \racts g.
}
\]
We say a commutative diagram is equivariant if each object is equipped with an endomorphism and the morphisms are equivariant with respect to these 
endomorphisms.
\item Let $(X, \Delta)$ be a klt pair where $X$ is a  normal projective variety. A $(K_{X}+ \Delta)$-negative extremal ray contraction $\pi \colon X  \longrightarrow Y$
is called of fiber type or a $(K_{X}+ \Delta)$-Mori fiber space if $\dim Y< \dim X$. We say simply Mori fiber space if $ \Delta=0$.
\item Let $M$ be a $\Z$-module. We write $M_{\Q}=M {\otimes}_{\Z}\Q$, $M_{\R}=M {\otimes}_{\Z}\R$, and so on.
\end{itemize}

\begin{ack}
The authors would like to thank the organizers of ``Younger generations in Algebraic and Complex geometry VI'' where
this collaboration started.
The first author would like to thank Sheng Meng and De-Qi Zhang for stimulating discussions.
The first author is supported by JSPS Research Fellowship for Young Scientists and KAKENHI Grant Number 18J11260.
The second author is supported by  the Program for Leading Graduate Schools, MEXT, Japan.
\end{ack}

\section{Arithmetic degree and Kawaguchi-Silverman  Conjectures}\label{sec:def}

\subsection{Dynamical degrees}

Let $X$ be a smooth projective variety defined over an algebraically closed field of characteristic zero
and $f \colon X \dashrightarrow X$ a  dominant rational map.
We define pull-back $f^{*} \colon N^{1}(X) \longrightarrow N^{1}(X)$ as follows.
Take a resolution of indeterminacy $\pi \colon X' \longrightarrow X$ of $f$ with $X'$ smooth projective.
Write $f'=f \circ \pi $.
Then we define $f^{*}=\pi_{*}\circ {f'}^{*}$.
This is independent of the choice of resolution.

\begin{defn}

Fix a norm $||\ ||$ on the finite dimensional real vector space $\End(N^{1}(X)_{\R})$.
Then the (first) dynamical degree of $f$ is 
\[
\delta_{f}=\lim_{n \to \infty}||(f^{n})^{*}||^{1/n}.
\]
This is independent of the choice of $||\ ||$.
We refer to \cite{dang, df, tru0}, \cite[\S 4]{ds} for basic properties of dynamical degrees.
\end{defn}

\begin{rmk}
The dynamical degree has another equivalent definition in terms of intersection numbers:
\[
\delta_{f}=\lim_{n \to \infty} ((f^{n})^{*}H\cdot H^{\dim X-1})^{1/n}
\]
where $H$ is any nef and big Cartier divisor on $X$ (cf. \cite{dang}).
\end{rmk}

\begin{rmk}\label{rmk on dyn deg}\ 
Dynamical degree is a birational invariant. That is, if $\pi \colon X \dashrightarrow X'$ is a birational map
and $f \colon X \dashrightarrow X$ and $f' \colon X' \dashrightarrow X'$ are conjugate by $\pi$, then $\delta_{f}=\delta_{f'}$.
In particular, we can define the dynamical degree of a self-dominant rational map of quasi-projective varieties by taking
a smooth projective model.
\end{rmk}

\begin{rmk}
If $X$ is a normal projective variety and $f \colon X \longrightarrow X$ is a surjective endomorphism, 
then $ \delta_{f}$ is equal to the spectral radius of $f^{*} \colon N^{1}(X) \longrightarrow N^{1}(X) $ (cf.  \cite[Lemma 3.1]{ma2}). 
By projection formula, $\delta_{f}$ is also equal to the spectral radius of $f_{*} \colon N_{1}(X) \longrightarrow N_{1}(X)$.
\end{rmk}

\subsection{Arithmetic degrees}

In this subsection, the ground field is $ \overline{\mathbb Q}$.
We briefly recall the definition of Weil height function.
Standard references for Weil height functions are \cite{bg,hs,Lan}, for example.

The height function on a projective space $\P^{N}(\QQ)$ is
\begin{align*}
\P^{N}(\QQ) \longrightarrow \R\ ;\  (x_{0}:\cdots:x_{N}) \mapsto \frac{1}{[K : \Q]}\sum_{v} \log \max\{|x_{0}|_{v}, \dots , |x_{N}|_{v}\}
\end{align*}
where $K$ is a number field (finite extension of $\Q$ contained in the fixed algebraic closure $\QQ$) containing the coordinates $x_{0},\dots,x_{N}$,
the sum runs over all places $v$ of $K$, and $|\ |_{v}$ is the absolute value associated with $v$ normalized as follows:
\begin{align*}
|x|_{v}= 
\begin{cases}
\#\ ( \mathcal{O}_{K}/  \mathfrak{p}_{v})^{-\ord_{v}(x)} 
\ \text{if $v$ is non-archimedian}\\
| \sigma_{v}(x)|^{[K_{v}:\R]}\ \text{if $v$ is archimedian.}
\end{cases}
\end{align*}
Here $ \mathcal{O}_{K}$ is the ring of integers of $K$.
When $v$ is non-archimedian, $ \mathfrak{p}_{v}$ is the maximal ideal corresponding to $v$ and $\ord_{v}$ is the valuation associated with $v$.
When $v$ is archimedian, $ \sigma_{v}$ is the embedding of $K$ into $\C$ corresponding to $v$.
This definition is independent of the choice of homogeneous coordinates and the number field $K$.

Let $X$ be a projective variety over $\QQ$.
An $\R$-Cartier divisor  $D$ on $X$ determines a (logarithmic) Weil height function $h_{D}$ up to bounded functions as follows.
When $D$ is a very ample integral divisor, $h_{D}$ is the composite of the embedding by the complete linear system
$|D|$ and the height on the projective space we have just defined.
For a general $D$, we write 

\begin{align}
\label{very ample sum}
D=\sum_{i=1}^{m}a_{i}H_{i} 
\end{align}
where $a_{i}$ are real numbers and $H_{i}$ are very ample divisors.
Then we define
\[
h_{D}=\sum_{i=1}^{m}a_{i}h_{H_{i}}.
\]
The function $h_{D}$ does not depend on the choice of the representation (\ref{very ample sum})
up to bounded function.
We call any representative of the class $h_{D}$ modulo bounded functions a height function
associated with $D$.

\begin{defn}
Let $X$ be a normal projective variety defined over $\QQ$.
Let $f \colon X \longrightarrow X$ be a surjective endomorphism. 
Let $H$ be an ample divisor on $X$. Fix a Weil height function $h_{H}$ associated with $H$.
The arithmetic degree $ \alpha_{f}(x)$ of $f$ at $x \in X(\QQ)$ is defined by
\[
\alpha_{f}(x)=\lim_{n \to \infty}\max\{1, h_{H}(f^{n}(x))\}^{1/n}.
\]
The definition of the arithmetic degree is independent of the choice of $H$ and $h_{H}$ (\cite[Proposition 12]{ks3} \cite[Theorem 3.4]{mss}). 
The existence of the limit is proved in \cite{ks1}.
\end{defn}

\begin{rmk}\label{rmk:exist-of-ad}
In \cite{ks1}, it is proved that $ \alpha_{f}(x)$ is equal to the absolute value of one of the eigenvalues of $f^{*} \colon N^{1}(X) \longrightarrow N^{1}(X)$ (\cite{ks1}).
In particular, $ \alpha_{f}(x) \leq \delta_{f}$ for all $x\in X(\QQ)$.
\end{rmk}

In \cite{ks3}, Kawaguchi and Silverman formulated the following conjecture.

\begin{conj}[KSC]\label{ksc}
Let $X$ be a normal projective variety and $f \colon X \longrightarrow X$ a surjective morphism,
both defined over $ \overline{\mathbb Q}$.
Let $x \in X( \overline{\mathbb Q})$.
If the orbit $ O_{f}(x)=\{ f^{n}(x) \mid n=0,1,2, \dots \}$ is Zariski dense in $X$,
then $\alpha_{f}(x)=\delta_{f}$.
\end{conj}

\begin{rmk}
In \cite{ks3}, the conjecture is actually formulated for dominant rational self-maps of  smooth projective varieties.
\end{rmk}

\begin{rmk}\label{iterksc}
Let $f \colon X \longrightarrow X$ be as in Conjecture \ref{ksc} and let $x \in X(\QQ)$.
Then, $\alpha_{f^{n}}(x)= \alpha_{f}(x)^{n}$ for any positive integer $n$ (\cite[Corollary 3.4]{ma}).
On the other hand, $\delta_{f^{n}}=\delta_{f}^{n}$ by definition.
Thus, we may replace $f$ with its power to prove Conjecture \ref{ksc}.
(Note that $O_{f}(x)$ is Zariski dense if and only if so is $O_{f^{n}}(x)$.)
\end{rmk}

\begin{rmk}\label{birinvksc}
Suppose $f \colon X \longrightarrow X$ and $g \colon Y \longrightarrow Y$ are surjective endomorphisms on normal projective varieties
and $\pi \colon X \dashrightarrow Y$ is a birational map such that $g\circ \pi = \pi \circ f$.
Then KSC for $f$ is equivalent to that of $g$ (cf. \cite[Lemma 5.6]{ma2}).
\end{rmk}

\begin{rmk}\label{baseksc}
Suppose $f \colon X \longrightarrow X$ and $g \colon Y \longrightarrow Y$ are surjective endomorphisms on normal projective varieties
and $\pi \colon X \dashrightarrow Y$ is a surjective morphism such that $g\circ \pi = \pi \circ f$.
Suppose furthermore that $ \delta_{f}= \delta_{g}$.
If KSC for $g$ holds, then KSC for $f$ also holds. 
Indeed, it is easy to see that $ \alpha_{f}(x)\geq \alpha_{g}(\pi(x))$ for every $x\in X(\QQ)$.
If $x\in X(\QQ)$ has Zariski dense $f$-orbit, then $\pi(x)$ has Zariski dense $g$-orbit and we get 
$ \alpha_{f}(x)\geq \alpha_{g}(\pi(x))= \delta_{g} = \delta_{f}$.
This implies $ \alpha_{f}(x)= \delta_{f}$.
\end{rmk}

\begin{rmk}
Conjecture \ref{ksc} is verified in several cases (not only for endomorphisms, but also for dominant rational maps).
See \cite[Remark 1.8]{mss},  \cite{ms, ls, ma2, mz, lin, jonwul}.
\end{rmk}

We use the following later.

\begin{prop}[{\cite[Proposition 3.6]{ma2}}]\label{canht-posiitaka}
Let $X$ be a normal projective variety and $f \colon X \longrightarrow X$ a surjective endomorphism, both defined over $\QQ$.
Suppose there exists a $\Q$-Cartier divisor $D$ on $X$ such that
\begin{enumerate}
\item\label{cond:eigen} $f^{*}D \sim_{\Q} dD$ and $d=\delta_{f}>1$;
\item\label{cond:iitaka} $\kappa(D)>0$.
\end{enumerate}
Then we have $ \alpha_{f}(x)=\delta_{f}$ for every $x\in X(\QQ)$ with Zariski dense $f$-orbit.
\end{prop}

\section{Equivariant MMP}\label{sec:eqmmp}

Meng and Zhang established minimal model program equivariant with respect to endomorphisms, for varieties admitting an int-amplified endomorphism.
In this section, we summarize their results that we need later.
We refer to \cite{kolmor} for the basic notions in minimal model program.

\begin{defn}\label{def:intamp}
A surjective endomorphism $f \colon X \longrightarrow X$ of normal projective variety $X$ is called int-amplified if 
there exists an ample Cartier divisor $H$ on $X$ such that $f^{*}H-H$ is ample.
\end{defn}

We collect basic properties of int-amplified endomorphisms in the following lemma.

\begin{lem}\label{lem:intamp}\ 
\begin{enumerate}
\item Let $X$ be a normal projective variety, $f \colon X \longrightarrow X$ a surjective morphism, and $n>0$ a positive integer.
Then, $f$ is int-amplified if and only if so is $f^{n}$.

\item Let $\pi \colon X \longrightarrow Y$ be a surjective morphism between normal projective varieties.
Let $f \colon X \longrightarrow X$, $g \colon Y \longrightarrow Y$ be surjective endomorphisms such that $\pi \circ f=g \circ \pi$.
If $f$ is int-amplified, then so is $g$. 

\item  Let $\pi \colon X \dashrightarrow Y$ be a dominant rational map between normal projective varieties of same dimension.
Let $f \colon X \longrightarrow X$, $g \colon Y \longrightarrow Y$ be surjective endomorphisms such that $\pi \circ f=g \circ \pi$.
Then $f$ is int-amplified if and only if so is $g$.

\item If a normal $\Q$-Gorenstein (i.e. $K_{X}$ is $\Q$-Cartier) projective variety $X$ admits an int-amplified endomorphism, then the anti-canonical divisor 
$-K_{X}$ is numerically equivalent to an effective $\Q$-divisor.

\end{enumerate}
\end{lem}
\begin{proof}
See \cite[Lemma 3.3, 3,5, 3.6, Theorem 1.5]{meng}.
\end{proof}


\begin{thm}[Meng-Zhang]\label{equiv-thm}
Let $X$ be a $\Q$-factorial normal projective variety over $\QQ$ admitting an int-amplified endomorphism.
Let $ \Delta$ be an effective $\Q$-divisor on $X$ such that $(X, \Delta)$ is a klt pair.
\begin{enumerate}
\item\label{fin-st} There are only finitely many $(K_{X}+ \Delta)$-negative extremal rays of $ \overline{NE}(X)$.
Moreover, let $f \colon X \longrightarrow X$ be a surjective endomorphism of $X$.
Then every $(K_{X}+ \Delta)$-negative extremal ray is fixed by the linear map $(f^{n})_{*}$ for some $n>0$.
\item\label{end-induced}  Let $f \colon X \longrightarrow X$ be a surjective endomorphism of $X$.
Let $R$ be a $(K_{X}+ \Delta)$-negative extremal ray and $\pi \colon X \longrightarrow Y$ its contraction.
Suppose $f_{*}(R)=R$.
Then,
\begin{enumerate}
\item \label{induced-on-contr} $f$ induces an endomorphism $g \colon Y \longrightarrow Y$ such that $g\circ \pi=\pi \circ f$;
\item \label{induced-on-flip}if $\pi$ is a flipping contraction and $X^{+}$ is the flip, the induced rational self-map $h \colon X^{+} \dashrightarrow X^{+}$
is a morphism.
\end{enumerate}
\item  In particular, for any finite sequence of $(K_{X}+ \Delta)$-MMP and for any surjective endomorphism $f \colon X \longrightarrow X$,
there exists a positive integer $n>0$ such that the sequence of MMP is equivariant under $f^{n}$.
\end{enumerate}
\end{thm}
\begin{proof}
(\ref{fin-st}) is a special case of  \cite[Theorem 4.6]{meng-zhang2}.
(\ref{induced-on-contr}) is true since the contraction is determined by the ray $R$.
(\ref{induced-on-flip}) follows from \cite[Lemma 6.6]{meng-zhang}.
\end{proof}

\begin{thm}[Equivariant MMP (Meng-Zhang)]\label{equiMMP}
Let $X$ be a $\Q$-factorial klt projective variety over $\QQ$ admitting an int-amplified endomorphism.
Then for any surjective endomorphism $f \colon X \longrightarrow X$, there exists a positive integer $n>0$ and 
a sequence of rational maps
\begin{align*}
X=X_{0} \dashrightarrow X_{1} \dashrightarrow \cdots \dashrightarrow X_{r}
\end{align*}
such that
\begin{enumerate}
\item $X_{i} \dashrightarrow X_{i+1}$ is either the divisorial contraction, flip, or Fano contraction of a 
$K_{X_{i}}$-negative extremal ray;
\item there exist surjective endomorphisms $g_{i} \colon X_{i} \longrightarrow X_{i}$ for $i=0, \dots ,r$
such that $g_{0}=f^{n}$ and the following diagram commutes:
\[
\xymatrix{
X_{i} \ar@{-->}[r] \ar[d]_{g_{i}} & X_{i+1} \ar[d]^{g_{i+1}}\\
X_{i} \ar@{-->}[r] & X_{i+1} ;
}
\]
\item $X_{r}$ is a $Q$-abelian variety (i.e. there exists a quasi-\'etale finite surjective morphism $A \longrightarrow X_{r}$ from an abelian variety $A$.
Note that $X_{r}$ might be a point.)
In this case, there exists a quasi-\'etale finite surjective morphism $A \longrightarrow X_{r}$ from an abelian variety $A$
and an surjective endomorphism $h \colon A \longrightarrow A$ such that the diagram
\[
\xymatrix{
A \ar[r]^{h} \ar[d] & A \ar[d]\\
X_{r} \ar[r]_{g_{r}} & X_{r} ;
}
\]
commutes;
\end{enumerate}
\end{thm}
\begin{proof}
This is a part of \cite[Theorem 1.2]{meng-zhang2}.
\end{proof}

\begin{rmk}
Surjective endomorphisms on a Q-abelian variety always lift to a certain quasi-\'etale cover by an abelian variety.
See \cite[Lemma 8.1 and Corollary 8.2]{cmz}, for example. The proof works over any algebraically closed field.
\end{rmk}

\section{Covering Theorem}\label{sec:covthm}

In this section, we prove that a certain type of dynamical system admits a quasi-\'etale cover which dominates an abelian variety.
We also prove that such covering ascends along processes of equivariant MMP.

\subsection{Main Components}

First, we recall basic facts on main components of fiber products.

\begin{lem}\label{main comp}
Consider the following commutative diagram:
\[
\xymatrix{
\widetilde{X} \ar[r]^(.4){ \alpha} \ar@/^18pt/[rr]^{ \widetilde{\pi}} \ar[rd]_{\mu_{X}} & X\times_{Z}A  \ar[r]^(.6){q} \ar[d]_{p} & A \ar[d]^{\mu_{Z}}\\
& X \ar[r]_{\pi} & Z
}
\]
where $X, Z, A, \widetilde{X}$ are normal projective varieties, $\mu_{X}, \mu_{Z}$ are finite surjective morphisms,
$ \widetilde{\pi}$ is a surjective morphism,  and $\pi$ is an algebraic fiber space.
Let $U \subset Z$ be a dense open subset such that $\mu_{Z}^{-1}(U) \longrightarrow U$ is \'etale.
Then the open subscheme $p^{-1}(\pi^{-1}(U)) \subset X\times_{Z}A$ is an irreducible normal variety.
The closure $M$ of this subset in $X\times_{Z}A$ is the unique irreducible component of $X\times_{Z}A$ that dominates $X$.
We call this $M$ the main component of $X\times_{Z}A$ and equip it with the reduced structure.

Moreover,
$ \alpha$ is the normalization of $M$ if and only if the canonical homomorphism $k(X) {\otimes}_{ k(Z)} k(A) \longrightarrow k( \widetilde{ X})$ is 
an isomorphism.

If this is the case and
$X, Z, A$ are equipped with surjective endomorphisms equivariantly, then they induce a surjective endomorphism on $ \widetilde{X}$.
\end{lem}

\begin{rmk}\label{maincompfunc}
In the setting of Lemma \ref{main comp}, the normalization of the main component is equal to the normalization of
$X$ in $k(X) {\otimes}_{k(Z)}k(A)$.
Note that since $k(X)\supset k(Z)$ is algebraically closed and $k(A) \supset k(Z)$ is a finite separable extension,
$k(X) {\otimes}_{k(Z)}k(A)$ is a field.
\end{rmk}

\begin{proof}[Proof of Lemma \ref{main comp}]
Since $\mu_{Z}^{-1}(U) \longrightarrow U$ is flat, $p^{-1}(\pi^{-1}(U)) \longrightarrow \pi^{-1}(U)$ is an open map 
and $p^{-1}(\pi^{-1}(U)) \longrightarrow \mu_{Z}^{-1}(U)$ is an algebraic fiber space.
This implies $p^{-1}(\pi^{-1}(U))$ is irreducible.
Since $p^{-1}(\pi^{-1}(U)) \longrightarrow \pi^{-1}(U)$ is \'etale and $X$ is normal, 
$p^{-1}(\pi^{-1}(U))$ is a normal variety.
This also proves that the uniqueness of the irreducible component that dominates $X$.

Since $k(X) {\otimes}_{k(Z)}k(A)$ is a field, the function field of $M$ is $k(X) {\otimes}_{k(Z)}k(A)$.
Therefore, $ \alpha$ is the normalization of $M$ if and only if $k(X) {\otimes}_{ k(Z)} k(A) \longrightarrow k( \widetilde{ X})$ is an 
isomorphism.

The last statement follows from the universality of the fiber products and the uniqueness of the main component.
\end{proof}

\begin{lem}\label{maincomp2}
Consider the following commutative diagram:
\[
\xymatrix{
\widetilde{X} \ar[r] \ar[d] & \widetilde{Y} \ar[r] \ar[d] & A \ar[d]^{\mu} \\
X \ar[r]_{p} & Y \ar[r]_{q} & Z
}
\]
where $X, Y, Z, \widetilde{X}, \widetilde{Y}$, and $A$ are normal projective varieties,
$p, q$ are algebraic fiber spaces, and $\mu$ is a finite surjective morphism.
Suppose that $ \widetilde{Y}$ is the normalization of the main component of $Y \times_{Z} A$ and
$ \widetilde{X}$ is the normalization of the main component of $X \times_{Y} \widetilde{Y}$.
Then $ \widetilde{X}$ is the normalization of the main component of $X \times_{Z}A$.
\end{lem}
\begin{proof}
By Lemma \ref{main comp}, $k( \widetilde{Y})=k(Y) {\otimes}_{k(Z)} k(A)$ and
$k( \widetilde{X})= k(X) {\otimes}_{k(Y)} k( \widetilde{Y})$.
Therefore, we have $k( \widetilde{X})=k(X) {\otimes}_{k(Z)} k(A)$.
By Lemma \ref{main comp} again, $ \widetilde{X}$ is the normalization of the main component of $X\times_{Z}A$.
\end{proof}

\subsection{Covering Theorem} 

The following theorem is taken form an upcoming paper "Globally $F$-splitness of surfaces admitting an int-amplified endomorphism" by the second author.
We use this theorem in the proof of the main theorem to show certain kind of fibrations do not appear during the process of equivariant MMP.

\begin{thm}\label{covlem}
Consider the following commutative diagram:
\[
\xymatrix{
X \ar[r]^{f} \ar[d]_{\pi} & X \ar[d]^{\pi} \\
Y \ar[r]_{g} & Y 
}
\]
where $X$ is a klt $\Q$-factorial normal projective variety, $\pi$ is a $K_{X}$-negative extremal ray contraction of fiber type,
$f$ is an int-amplified endomorphism.
Assume for all irreducible component $E\subset \Supp R_{f}$, we have $\pi(E)=Y$.  
Then there exists a following equivariant  commutative  diagram: 
\[
\xymatrix{
f \acts X \ar@<2.3ex>[d]_{\pi}& \ar[l]_{\mu_{X}} \widetilde{X} \ar@<-2.3ex>[d]^{ \widetilde{\pi}} \racts \widetilde{f} \\
g \acts Y & \ar[l]^{\mu_{Y}} A \racts g_{A}
}
\]
where $A$ is an abelian variety, $\mu_{Y}$ is a finite surjective morphism, $ \widetilde{X}$ is the normalization of the main component of $X {\times}_{Y}A$
satisfying the following:
\begin{itemize}
\item $ \widetilde{X}$ is a $\Q$-Gorenstein klt normal projective variety;
\item $\mu_{X}$ is a finite surjective quasi-\'etale morphism;
\item $ \widetilde{\pi}$ is an algebraic fiber space.
\end{itemize}
\end{thm}

We prepare to prove the theorem.
First, we prove that if $Y$ has a finite $g$-equivariant cover by an abelian variety, then $\mu_X$ as in Theorem \ref{covlem} is quasi-\'etale.

\begin{lem}\label{qasi-etale lemma}
Let $X$ be a normal projective variety and $f \colon X \longrightarrow X$ an int-amplified endomorphism and assume that there exists the following commutative diagram:

\[
\xymatrix{
f \acts X \ar@<2.3ex>[d]_{\pi}& \ar[l]_{\mu_{X}} \widetilde{X} \ar@<-2.3ex>[d]^{ \widetilde{\pi}} \racts \widetilde{f} \\
g \acts Y & \ar[l]^{\mu_{Y}} \widetilde{Y} \racts \widetilde{g}
}
\]
where
\begin{enumerate}
\item $Y$ and $\widetilde{Y}$ are normal projective varieties; 
\item $\pi$ is an algebraic fiber space, $\mu_{X}$ and $\mu_{Y}$ are finite surjective morphisms;
\item $ \widetilde{X}$ is the normalization of the main component of $X \times_{Y} \widetilde{Y}$;
\item $g, \widetilde{g}, \widetilde{f}$ are int-amplified endomorphisms.
\end{enumerate}
Furthermore we assume the following conditions:
\begin{enumerate}
\renewcommand{\labelenumi}{(\alph{enumi})}
\item\label{tilgqet} $\widetilde{g}$ is quasi-\'etale;
\item\label{imEcodim} every prime divisor $E$ on $X$ satisfies $\mathrm{codim} (\pi(E)) \leq 1$;   
\item\label{ramhor} every irreducible component $E$ of $\Supp{R_f}$ satisfies $\pi(E) = Y$
\end{enumerate}
Then $\mu_X$ is quasi-\'etale.
\end{lem}

\begin{proof}
There exists a nonempty open subset $U$ of $Y$ such that $\mu_Y^{-1}(U) \longrightarrow U$ is \'etale.
Then since $\widetilde{\pi}^{-1} (\mu_Y^{-1}(U)) = \mu_X^{-1}(\pi^{-1}(U)) = \pi^{-1}(U) \times_{U} \mu_Y^{-1}(U)$ by Lemma \ref{main comp}, $\mu_X^{-1}(\pi^{-1}(U)) \longrightarrow \pi^{-1}(U)$ is \'etale.
In particular, every irreducible component $E$ of $\Supp(R_{\mu_X})$ satisfies $\widetilde{\pi}(E) \neq \widetilde{Y}$.
On the other hand, we have
\[
\widetilde{f}^*(R_{\mu_X}) \leq \widetilde{f}^*(R_{\mu_X}) + R_{\widetilde{f}} = \mu_X^*(R_f) + R_{\mu_X}
\]
Since any component of $\mu_X^*(R_f)$ is horizontal and any component of $R_{\mu_X}$ is vertical, we obtain
\[
\widetilde{f}^*(R_{\mu_X}) \leq R_{\mu_X}
\]
as Weil divisors.
It means that $\Supp{R_{\mu_X}}$ is a totally invariant divisor under $\widetilde{f}$, 
therefore, by Lemma \ref{totinvim}, $\widetilde{\pi}(\Supp{R_{\mu_X}})$ is a totally invariant closed subset under $\widetilde{g}$.
Suppose that $\mu_X$ is not quasi-\'etale.
Let $E$ be an irreducible componet of $\Supp{R_{\mu_X}}$, then $E$ is a prime divisor.
Replacing $\widetilde{g}$ by some iterate, we may assume that $\widetilde{\pi}(E)$ is totally invariant.
By the assumption, $\mathrm{codim}(\widetilde{\pi}(E))$ is less than or equal to $1$.
Since $E$ is a vertical divisor, $\widetilde{\pi}(E)$ must be a prime divisor on $\widetilde{Y}$.
It contradicts to the fact that $\widetilde{g}$ is an int-amplified quasi-\'etale endomorphism (cf. \cite[Lemma 3.11]{meng}).
Hence $\mu_X$ is quasi-\'etale. 
\end{proof}

\begin{lem}\label{totinvim}
Consider the following commutative diagram:
\[
\xymatrix{
X \ar[r]^{f} \ar[d]_{\pi} & X \ar[d]^{\pi} \\
Z \ar[r]_{h}  & Z 
}
\]
where $X, Z$ are normal projective varieties, $\pi$ is an algebraic fiber space, and
$f, h$ are surjective endomorphism.
Let $W\subset X$ be a closed subset.
If $f^{-1}(W) = W$ as sets, then $h^{-1}(\pi(W))=\pi(W)$.
\end{lem}
\begin{proof}
Since $h(\pi(W))=\pi(f(W))\subset \pi(W)$, we have $\pi(W)\subset h^{-1}(\pi(W))$.
Take a closed point $z\in h^{-1}(\pi(W))$.
Since $X$ is a normal variety and $f$ is a finite morphism, $f$ is an open map by going down.
Therefore, the induced map $\pi^{-1}(h^{-1}(h(z))) \longrightarrow \pi^{-1}(h(z))$ is also an open map.
Thus $f(\pi^{-1}(z))$ is an open and closed subset of $\pi^{-1}(h(z))$ and therefore $f(\pi^{-1}(z))=\pi^{-1}(h(z))$.
Since $\pi^{-1}(h(z))\cap W \neq \emptyset$, we get $\emptyset \neq \pi^{-1}(z)\cap f^{-1}(W) =  \pi^{-1}(z)\cap W$, which means
$z\in \pi(W)$.
\end{proof}

Next, we recall the following basic property of Mori fiber spaces.

\begin{lem}\label{irreducible fiber lemma}
Let $(X, \Delta)$ be a klt pair where $X$ is a $\Q$-factorial normal projective variety and $ \Delta$ is an effective $\Q$-divisor.
Let $\pi \colon X \longrightarrow Y$ be a $(K_{X}+ \Delta)$-negative extremal ray contraction of fiber type.
Then for any prime divisor $E$ on $Y$, $\pi^{-1}(E)$ is an irreducible codimension one closed subset of $X$.
In particular, any prime divisor $F$ on $X$ satisfies $\mathrm{codim}(\pi(F)) \leq 1 $.
\end{lem}

\begin{proof}
Suppose that $\pi^{-1}(E)$ has two components.
Then we can take irreducible components $F$ and $F'$ such that $\mathrm{codim}(F) = 1$, $\pi(F)=E$ and $F \cap F' \neq \emptyset$. 
Take a closed point $y \in \pi(F' \backslash F)$, then there exist closed points $x \in F$ and $x' \in F' \backslash F $ such that $\pi(x)=\pi(x')=y$.
Since $\pi$ has connected fibers, we can take a connected curve containing $x$ and $x'$, so we can take an integral curve $C$ such that $C \cap F \neq \emptyset $ and C is not contained in $F$. 
In particular, the intersection number of $C$ and $F$ is positive.
However, taking a closed point $y' \in Y \backslash E$ and an integral curve $C' \subset \pi^{-1}(y')$ as sets, we have $(C'\cdot F) = 0 $.
Since $\pi$ is a $(K_{X}+ \Delta)$-Mori fiber space, $C$ and $C'$ are proportional in $N_{1}(X)_{\R}$. Hence this is a contradiction.

Next, we assume that there exists a prime divisor $F$ on $X$ such that $\mathrm{codim}(\pi(F)) \geq 2$.
Then we can take a prime divisor $E$ on $Y$ containing $\pi(F)$.
Hence we have $F \subset \pi^{-1}(E)$ as sets. By the first assertion, $F = \pi^{-1}(E)$ as sets.
In particular, $\pi(F) = E$ and it is contradiction.

\end{proof}

\begin{defn}
Let $X$ be a normal projective variety and $\Delta$ an effective $\Q$-Weil divisor on $X$.
We say that $\Delta$ has standard coefficients if for any prime divisor $E$ on $X$, there exists a positive integer $m$ such that $\ord_{E}(\Delta) = \frac{m-1}{m}$.

\end{defn}

Next, we consider the index-1 cover of log trivial pairs.
Index-1 cover is not quasi-\'etale when the boundary is not $\Z$-Weil.
However, if the boundary has standard coefficients,
we can compute the ramification indices and conclude that the canonical divisor of index-1 cover is a principal divisor.

\begin{lem}\label{index-one covering lem}
Let $X$ be a normal projective variety and $\Delta$ an effective $\Q$-Weil divisor on $X$ such that $K_X + \Delta $ is $\Q$-linearly equivalent to $0$.
Then there exists a finite surjective morphism $\mu \colon \widetilde{X} \longrightarrow X$ from a normal projective variety such that the following conditions hold:
\begin{itemize}
    \item $\mu^*(K_X +\Delta)$ is a principal divisor: $\mu^*(K_X+\Delta) \sim 0$;
    \item if  $\mu' \colon X' \longrightarrow X$ is a finite surjective morphism from a normal projective variety such that $\mu'^*(K_X+\Delta)$ is a principal divisor, then $\mu'$ factors through $\mu$.
\end{itemize}
Furthermore, if $\Delta$ has standard coefficients, then $R_{\mu} = \mu^*(\Delta)$.
In particular $K_{\widetilde{X}}$ is a principal divisor.

\end{lem}

\begin{proof}
Let $m_0 = \min\{m\  |\  m(K_X+\Delta)\sim 0\}$ and take a non-zero rational section $\alpha \in k(X) =K$ such that $\mathrm{div}_{X}(\alpha)=m_0(K_X+\Delta)$.
Let $L=K[T]/(T^{m_0}-\alpha)$.
Note that $L$ is a field.
Let $\mu \colon \widetilde{X} \longrightarrow X$ be the normalization of $X$ in $L$.
Then we have
\[
\mathrm{div}_{ \widetilde{X}}(T)=\frac{1}{m_0}\mathrm{div}_{ \widetilde{X}}(\alpha) = \mu^*(K_X+\Delta),
\]
so $\mu^*(K_X+\Delta)$ is a principal divisor.
Moreover, let $\mu' \colon X' \longrightarrow X$ be a finite surjective morphism from a normal projective variety such that $\mu'^*(K_X+\Delta)$ is a principal divisor.
Then there exists a non-zero rational function $\beta \in K(X')=L'$ such that $\mathrm{div}_{X'}(\beta)=\mu'^*(K_X+\Delta)$. 
In particular, we have $m_0 \mathrm{div}_{X'}(\beta) = \mathrm{div}_{X'}(\alpha)$.
Since the base field is algebraically closed, we may assume that $\beta^{m_0}=\alpha$.
It means that there exists an injective $K$-algebra homomorphism from $L$ to $L'$,
so $\mu'$ factors through $\mu$.

Next we assume that $\Delta$ has standard coefficients.
Let $E$ be a prime divisor on $X$, $m$ a positive integer, $a$ an integer such that $\ord_E(K_X+\Delta)=\frac{m-1}{m}+a$.
Let $(R,(\varpi))$ be the DVR associated to $E$ and $S$ the normalization of $R$ in $L$.
Then it is enough to show that the order of $\varpi$ at any maximal ideal of $S$ is equal to $m$.
Since the order of $\alpha$ along $E$ is equal to $m_0(\frac{m-1}{m}+a)$, there exists an unit $u$ in $R$ such that
\[
\alpha=u \varpi^{m_0(\frac{m-1}{m}+a)}.
\]
Since every coefficients of $\mathrm{div}_{X}(\alpha)$ are integer, there exists a positive integer $b$ such that $mb=m_0$.
Hence if we set $\rho=m-1+ma$, then 
\[
\alpha=u \varpi^{b\rho}.
\]
Now, we have
\[
(\frac{T^m}{\varpi^{\rho}})^b = \frac{\alpha}{\varpi^{b\rho}} = u,
\]
so $S$ contains $\frac{T^m}{\varpi^{\rho}}$.
In particular, $R \hookrightarrow S$ factors through $R'=R[Y]/(Y^b-u)$.
Since $K'=K[Y]/(Y^b-u)$ is a field and $R' \longrightarrow K'$ is injective,
$R'$ is an integral domain and satisfies $R \subset R' \subset S$.
Since $R'$ is \'etale over $R$, $\varpi$ is an uniformizer of $R'_{i}=R'_{\mathfrak{p}_i}$ for any maximal ideal $\mathfrak{p_i}$ of $R'$.
We set $S_i = S \otimes_{R'} R'_i \subset L$.
Now, we have
\[
(\frac{\varpi^{1+a}}{T})^m=\frac{\varpi^\rho}{T^m} \cdot \varpi = Y^{-1}\varpi.
\]
Since $Y$ is an unit in $R'$, $R_i' \hookrightarrow S_i$ factors through $R''_i=R'_i[Z]/(Z^m-Y^{-1}\varpi)$.
Because $K''=K'[Z]/(Z^m-Y^{-1}\varpi)$ is a field and $R''_i \longrightarrow K''$ is injective,
$R''_i$ is an integral domain and satisfies $R'_i \subset R''_i \subset S_i$.
Since we have
\[
(\varpi^{1+a}Z^{-1})^{m_0}=(\varpi^{(1+a)m}Z^{-m})^b=(\varpi^{\rho}Y)^b=\alpha,
\]
the quotient field of $R''_i$ is $L$.
Furthermore, since $(Z)$ is the unique maximal ideal of $R''_i$, we obtain $R''_i=S_i$ and $\ord_{S_i}(\varpi)=m$.
Therefore we obtain the last assertion.
\end{proof}

The following lemma is well-known to experts.

\begin{lem}\label{canonical bundle formula lem}
Let $\pi \colon X \longrightarrow Y$ be an algebraic fiber space where $X, Y$ are normal $\Q$-Gorenstein projective varieties.
Assume that $-K_X$ is $\pi$-ample and $X$ is klt.
Then $Y$ is also klt.
\end{lem}

\begin{proof}
There exists an ample Cartier divisor $H$ on $Y$ such that $-K_X+\pi^*H$ is an ample $\Q$-Cartier divisor.
Then for a general effective $\Q$-Cartier divisor $B$ that is linearly equivalent to $-K_X+\pi^*H$, $(X,B)$ is a klt pair and $K_X+B \sim_{\Q} \pi^*H$.
By Ambro's canonical bundle formula (\cite[Theorem 4.1]{ambro}), we find an effective $\Q$-Weil divisor $B_Y$ such that $(Y,B_Y)$ is a klt $\Q$-Gorenstein pair.
In particular, $Y$ is klt.
\end{proof}

\begin{proof}[Proof of Theorem \ref{covlem}]
By Lemma \ref{irreducible fiber lemma}, we can define a positive integer $m_E$ for any prime divisor $E$ on $Y$ so that $\pi^*E = m_E F$ for some prime divisor $F$ on $X$.
We set $\Delta=\sum \frac{m_E-1}{m_E}E$ as a $\Q$-Weil divisor, where the sum runs over all the prime divisors on $Y$.
Note that this is a finite sum because a general fiber of $\pi$ is reduced. 
Thus $\Delta$ is a well-difined $\Q$-Weil divisor with standard coefficients.
First we prove that $K_Y+\Delta \sim_{\Q} 0$ and $K_Y+\Delta \sim g^*(K_Y+\Delta)$.
Let $E$ be a prime divisor on $Y$ and we set $g^*E=a_1 E_1+ \cdots + a_r E_r$, where $a_1 ,\ldots , a_r$ are positive integers and $E_1, \ldots E_r$ are prime divisors on $Y$.
For every $i$, we can write $\pi^*(E_i)=m_{E_i}F_i$ for some prime divisor $F_{i}$ on $X$.
Since $\pi^*g^*E=f^*\pi^*E$, we have
\[
a_1 m_{E_1} F_1 + \cdots +a_r m_{E_r} F_r = m_E(F_1+ \cdots F_r)
\]
and $a_i m_{E_i}=m_E$ for all $i$.
We remark that $f^*F = F_1 + \cdots F_r$ because any component of $R_f$ is horizontal.
Using this equality, we have
\[
\ord_{E_i}(g^*\Delta-\Delta)=a_i \frac{m_E-1}{m_E}-\frac{m_{E_i}-1}{m_{E_i}} =\frac{m_E-m_{E_i}}{m_{E_i}} = a_i-1 = \ord_{E_i}(R_g).
\]
Hence we obtain $g^*(K_Y+\Delta)\sim K_Y+\Delta$.
Since $g$ is an int-amplified endomorphism, $K_Y +\Delta$ is numerically equivalent to $0$.
Furthermore by \cite[Lemma 2.11]{brou-hor}, $g$ induces the endomorphism of the log canonical model of $(Y,\Delta)$.
By the same argument of the proof of \cite[Theorem 1.4]{brou-hor}, $(Y,\Delta)$ is a log canonical pair.
It means that $(Y,\Delta)$ is a log Calabi-Yau pair, and by \cite[Theorem 1.2]{gong}, $K_Y +\Delta$ is $\Q$-linearly equivalent to $0$.

Applying Lemma \ref{index-one covering lem} to the pair $(Y,\Delta)$, 
we obtain the finite covering $\mu_{Y_1} \colon Y_1 \longrightarrow Y$ as in Lemma \ref{index-one covering lem}.
Since 
\[
\mu_{Y_1}^*g^*(K_Y+\Delta) \sim \mu_{Y_1}^*(K_Y+\Delta) \sim 0,
\]
$g \circ \mu_{Y_1}$ factors $\mu_{Y_1}$.
It means that there exists $g_1 \colon Y_1 \longrightarrow Y_1$ such that the following diagram commutes:
\[
\xymatrix{
Y_1 \ar[r]^{g_1} \ar[d]_{\mu_{Y_1}} & Y_1 \ar[d]^{\mu_{Y_1}} \\
Y \ar[r]_{g} & Y. 
}
\]
Let $X_1$ be the normalization of the main component of $X \times_Y Y_1$.
By Lemma \ref{main comp}, $X_1$ has an int-amplified endomorphism and we get the the following equivariant commutative diagram:
\[
\xymatrix{
f \acts X \ar@<2.3ex>[d]_{\pi}& \ar[l]_{\mu_{X_1}} X_1 \ar@<-2.3ex>[d]^{\pi_1} \racts f_1 \\
g \acts Y & \ar[l]^{\mu_{Y_1}} Y_1 \racts g_{1}.
}
\]
By Lemma \ref{irreducible fiber lemma}, any prime divisor $E$ on $X$ satisfies $\mathrm{codim}(\pi(E)) \leq 1$.
Since $\Delta$ has standard coefficients, $K_{Y_1}$ is linearly equivalent to $0$, in particular, $g_1$ is quasi-\'etale.
Hence by Lemma \ref{qasi-etale lemma}, $\mu_{X_1}$ is quasi-\'etale.
It implies that $X_1$ is $\Q$-Gorenstein klt and $-K_{X_1}$ is $\pi_1$-ample.
Therefore by Lemma \ref{canonical bundle formula lem}, $Y_1$ is also klt.
By \cite[Theorem 5.2]{meng}, there exists the following commutative diagram:
\[
\xymatrix{
A \ar[r]^{g_A} \ar[d]_{\mu_{Y_2}} & A \ar[d]^{\mu_{Y_2}} \\
Y_1 \ar[r]_{g_1} & Y_1, 
}
\]
where $A$ is an abelian variety, $\mu_{Y_2}$ is a finite surjective morphism and $g_A$ is an int-amplified endomorphism.
Let $\widetilde{X}$ be the normalization of the main component of $X \times_Y A$.
Then $\widetilde{X}$ has an int-amplified endormophism by Lemma \ref{main comp}, and the following diagram commutes:
\[
\xymatrix{
f \acts X \ar@<2.3ex>[d]_{\pi}& \ar[l]_{\mu_{X}} \widetilde{X} \ar@<-2.3ex>[d]^{\widetilde{\pi}} \racts \widetilde{f} \\
g \acts Y & \ar[l]^{\mu_{Y}} A \racts g_A,
}
\]
where $\mu_Y = \mu_{Y_1} \circ \mu_{Y_2}$.
By Lemma \ref{qasi-etale lemma}, $\mu_X$ is quasi-\'etale.

\end{proof}

\subsection{Covering and MMP} 

In this subsection, we prove that such a covering constructed in Theorem \ref{covlem} ascends along the process of equivariant MMP.

\begin{defn}\label{starcond}
Let $X$ be a normal projective variety and $f \colon X \longrightarrow X$ a surjective endomorphism.
We say $f$ satisfies $(\ast)$ if the following holds.
There exits a following equivariant  commutative  diagram:
\[
\xymatrix{
\widetilde{f} \acts \widetilde{X} \ar@<2.3ex>[d]_{\mu_{X}} \ar[r]^{ \widetilde{\pi}} & A \ar@<-2.3ex>[d]^{\mu_{Z}} \racts g_{A} \\
f \acts X \ar[r]_{\pi} & Z \racts g 
}
\]
where
\begin{enumerate}
\item $Z$ is a normal projective variety of positive dimension, $A$ is an abelian variety; 
\item $\pi$ is an algebraic fiber space, $\mu_{Z}$ is a finite surjective morphism;
\item\label{tilXismaincomp} $ \widetilde{X}$ is the normalization of the main component of $X \times_{Z} A$;
\item $\mu_{X}$ is a finite surjective quasi-\'etale morphism, and $ \widetilde{\pi}$ is an algebraic fiber space
(these properties follow from (\ref{tilXismaincomp}) except the quasi-\'etaleness); 
\item $g, g_{A}, \widetilde{f}$ are surjective endomorphisms.
\end{enumerate}

\end{defn}

\begin{lem}\label{liftlem}
Let $(X, \Delta)$ be a klt pair, where $X$ is a $\Q$-factorial normal projective variety and  $\Delta$ is an effective $\Q$-divisor.
Let $f \colon X \longrightarrow X$ be an int-amplified endomorphism.
Consider the following commutative diagram:
\[
\xymatrix{
X \ar@{-->}[r]^{\pi} \ar[d]_{f} & Y \ar[d]^{g}\\
X \ar@{-->}[r]^{\pi} & Y 
}
\]
where $\pi$ is the divisorial contraction, flip, or fiber type contraction of a $(K_{X}+\Delta)$-negative extremal ray and
$g$ is an endomorphism.
If $g$ satisfies $(\ast)$, then so does $f$.
\end{lem}
\begin{proof}
Take a diagram as in Definition \ref{starcond} for $g$:
\[
\xymatrix{
\widetilde{g} \acts \widetilde{Y} \ar[r]^{ \widetilde{p}} \ar@<2.1ex>[d]_{\mu_{Y}} & A \racts h_{A} \ar@<-2.3ex>[d]^{\mu_{Z}}\\
g \acts Y \ar[r]_{p} & Z \racts h .
}
\]
Note that $g, h, h_{A}$, and $ \widetilde{g}$ are int-amplified endomorphisms since so is $f$.  

(1) Let $\pi$ be a divisorial contraction.
Let $ \widetilde{X}$ be the normalization of the main component of $X \times_{Y} \widetilde{Y}$ and
consider the following commutative diagram:
\[
\xymatrix{
\widetilde{X} \ar@/^18pt/[rr]^{ \widetilde{\pi}} \ar[r] \ar[rd]_{\mu_{X}} &  X \times_{Y} \widetilde{Y} \ar[d] \ar[r] & \widetilde{Y} \ar[d]_{\mu_{Y}} \ar[r]^{ \widetilde{p}} & A \ar[d]^{\mu_{Z}}\\
 & X \ar[r]_{\pi} & Y \ar[r]_{p} & Z.
}
\]
Since $\mu_{Y}$ is quasi-\'etale, it is clear that $\mu_{X}$ is \'etale at a codimension one point of $ \widetilde{X}$ whose image in $X$ is not the exceptional divisor of $\pi$.
By Lemma \ref{maincomp2}, $ \widetilde{X}$ is the normalization of the main component of 
$X \times_{Z} A$.
In particular, $f, h$, and $h_{A}$ induce an int-amplified endomorphism $ \widetilde{f}$ on $ \widetilde{X}$.
The only thing that we need to show is that $\mu_{X}$ is \'etale at every codimension one point $P \in \widetilde{X}$ over 
the generic point of the exceptional divisor $E$ of $\pi$.
Since $\pi$ is equivariant under $f$ and $g$, we have $f^{-1}(E)=E$ as sets.
By Lemma \ref{totinvim}, we have $h^{-1}(p(\pi(E)))=p(\pi(E))$.
Therefore, $\mu_{Z}^{-1}(p(\pi(E)))$ is also totally invariant under $h_{A}$. 
Since $h_{A}$ is \'etale and int-amplified, $\mu_{Z}^{-1}(p(\pi(E)))=A$ (cf. \cite[Lemma 4.1]{meng}). 
Thus $p(\pi(E))=Z$.
Consider the following commutative diagram:
\[
\xymatrix{
\widetilde{X} \ar[r]^(.4){\nu} \ar[rd]_{\mu_{X}} & X\times_{Z}A \ar[d]_{ \alpha} \ar[r] & A \ar[d]^{\mu_{Z}}  \\
 & X \ar[r]_{p\circ \pi}& Z 
}
\]
where $\nu$ is the normalization of the main component.
Then, $ \alpha$ is \'etale at $\nu(P)$.
Therefore, the main component of $X \times_{Z}A$ is normal at $\nu(P)$.
Thus, $\nu$ is isomorphic around $P$ and $\mu_{X}$ is \'etale at $P$.

(2) Let $\pi$ be a flip. 
Let $\mu_{X} \colon \widetilde{X} \longrightarrow X$ be the normalization of $X$ by $k(X)\simeq k(Y) \subset k( \widetilde{Y})$
where the isomorphism is the one induced by $\pi$ and the inclusion is the one induced by $\mu_{Y}$.
Then we get the following equivariant commutative diagram:
\[
\xymatrix@R=0.1pt{
& \widetilde{g} \ar@{}[ddd]|{\uacts}&\\
&&\\
&&\\
\qquad \widetilde{X} \ar@{-->}[r]^(0.6){ \widetilde{\pi}} \ar@<2.3ex>[ddddd]_{\mu_{X}} &  \widetilde{Y} \ar[r]^(.4){ \widetilde{p}} \ar[ddddd]_{\mu_{Y}} & A \racts h_{A} \ar@<-2.3ex>[ddddd]^{\mu_{Z}} \\
&&\\
&&\\
&&\\
&&\\
f \acts X \ar@{-->}[r]_(.6){\pi} & Y \ar[r]_(.4){p} & Z \racts h .\\
&&\\
&g \ar@{}[uu]|{\dacts}&
}
\]
Since $\pi$ is isomorphic in codimension one, so is $ \widetilde{\pi}$ by construction.
In particular, $\mu_{X}$ is quasi-\'etale.
This implies $ \widetilde{X}$ is $\Q$-Gorenstein klt.
In particular, $ \widetilde{X}$ has at worst rational singularity.
Thus, the rational map $ \widetilde{p}\circ \widetilde{\pi} \colon \widetilde{X} \dashrightarrow A$ is a morphism (cf. \cite[Lemma 5.1]{meng-zhang}). 
By Lemma \ref{fincovmor}, $p\circ \pi$ is also a morphism.
Looking at the function fields, $p\circ \pi$ is an algebraic fiber space and 
$ \widetilde{X}$ is the normalization of the main component of $X\times_{Z} A$.

(3) Let $\pi$ be a fiber type contraction.
Let $ \widetilde{X}$ be the normalization of the main component of $X \times_{Y} \widetilde{Y}$.
Then we get  the following equivariant commutative diagram:
\[
\xymatrix@R=0.1pt{
&& \widetilde{g} \ar@{}[ddd]|{\uacts} &\\
&&&\\
&&&\\
\widetilde{f} \acts \widetilde{X} \ar@/^20pt/[rr]^{ \widetilde{\pi}} \ar[rddddd]_{\mu_{X}} \ar[r]^{\beta} & X\times_{Y} \widetilde{Y} \ar[ddddd]_{ \alpha} \ar[r] & \widetilde{Y} \ar[ddddd]_{\mu_{Y}} \ar[r]^(.4){ \widetilde{p}} & A \racts h_{A} \ar@<-2.3ex>[ddddd]^{\mu_{Z}}\\
&&&\\
&&&\\
&&&\\
&&&\\
 &  X \ar[r]_{\pi} &  Y \ar[r]_(.4){p} & Z \racts h.\\
&&&\\
&f \ar@{}[uu]|{\dacts}&g \ar@{}[uu]|{\dacts}& 
}
\]
By Lemma \ref{maincomp2}, $ \widetilde{X}$ is the normalization of the main component of $X \times_{Z} A$.
It is enough to show that $\mu_{X}$ is quasi-\'etale.
Let $P\in \widetilde{X}$ be a codimension one point.
Since $\pi$ is a contraction of fiber type, $\pi(\mu_{X}(P))$ is either the generic point of $Y$ or
a codimension one point of $Y$ (Lemma \ref{irreducible fiber lemma}).
Therefore, $\mu_{Y}$ is \'etale at $ \widetilde{\pi}(P)$ and thus $ \alpha$ is \'etale at $\beta(P)$.
Thus the main component of $X \times_{Y} \widetilde{Y}$ is normal at $\beta(P)$ and $\beta$ is isomorphic around $P$.
This implies $\mu_{X}$ is \'etale at $P$.
\end{proof}

\begin{lem}\label{fincovmor}
Consider the following commutative diagram:
\[
\xymatrix{
\widetilde{X} \ar[rd]^{g} \ar[d]_{\nu} & \\
X \ar@{-->}[r]_{f} & Y 
}
\]
where $X, Y$, and $ \widetilde{X}$ are normal projective varieties,
$\nu$ is a finite surjective morphism, $g$ is a morphism, and $f$ is a rational map.
Then $f$ is a morphism.
\end{lem}
\begin{proof}
Let $ \Gamma_{f}$ and $ \Gamma_{g}$ be the graph of $f$ and $g$.
Then we have the following commutative diagram:
\[
\xymatrix{
\widetilde{X} \ar[d]_{\nu} & \widetilde{X} \times Y \ar[l] \ar[d]_{\nu \times \id} & \Gamma_{g} \ar[l] \ar[d]^{ \beta} \ar@/_18pt/[ll]_{ \alpha} \\
X  & X \times Y \ar[l] & \Gamma_{f} \ar[l] \ar@/^18pt/[ll]^{ \gamma}.
}
\]
Note that $ \alpha$ is an isomorphism, $\nu $ is a finite morphism, and $ \beta$ is a surjective morphism.
This implies $ \gamma$ is birational and finite, that is an isomorphism.
\end{proof}

\section{Proof of the main theorem}\label{sec:prf}

The following proposition, which is a key step to prove the main theorem, is taken from \cite{mz} with a slight change.

\begin{prop}[{cf. \cite[Proposition 9.2]{mz}}]\label{3alt}
Let $X$ be a $\Q$-factorial klt normal projective variety over $\QQ$ with $\Pic(X)_{\Q}=N^{1}(X)_{\Q}$ admitting an int-amplified endomorphism.
Consider the following commutative diagram:
\[
\xymatrix{
X \ar[r]^{f} \ar[d]_{\pi} & X \ar[d]^{\pi} \\
Y \ar[r]_{g} & Y
}
\]
where $\pi$ is a $K_{X}$-negative extremal ray contraction with $\dim Y< \dim X$.
Then one of the following holds:

\begin{enumerate}
\item\label{case:kscholds} KSC for $f$ holds.
\item\label{case:k=0} $ \delta_{f}> \delta_{g}$, $\kappa(-K_{X})=0$, and $\dim Y >0$. 
Moreover, take an effective $\Q$-divisor $D \sim_{\Q} -K_{X}$. Then $\Supp(D)$ is irreducible.
\item\label{case:mmp} There exists an effective $\Q$-divisor $ \Delta$ such that $(X, \Delta)$ is klt, $K_{X}+ \Delta$ is not pseudo-effective, and
there exits a sequence of $(K_{X}+ \Delta)$-MMP:
\[
\xymatrix{
X = X_{1} \ar@{-->}[r] & \cdots  \ar@{-->}[r] & X_{r} \ar[d]^{\pi'} \\
              & & Y'
}
\]
where the horizontal dotted arrows are birational maps and $\pi'$ is a $(K_{X_{r}}+ \Delta_{r})$-negative extremal ray contraction of fiber type
where $ \Delta_{r}$ is the strict transform of $ \Delta$,
satisfying either $\rho(X)>\rho(X_{r})$ or $ \delta_{f^{n}|_{X_{r}}} = \delta_{f^{n}|_{Y'}}$.
Here $n>0$ is any positive integer such that the above sequence of MMP is $f^{n}$-equivariant and
$f^{n}|_{X_{r}}$, $f^{n}|_{Y'}$ are the induced endomorphisms.
Note that $r$ could be $1$. 
\end{enumerate}
\end{prop}
\begin{proof}
If $ \delta_{f}= \delta_{g}$, then (\ref{case:mmp}) holds with $ \Delta=0$, $r=1$, and $\pi'=\pi$.
Suppose $\delta_{f}> \delta_{g}$.
If $\dim Y=0$, then $\rho(X)=1$ and (\ref{case:kscholds}) holds (cf. \cite{ks3}).  
Suppose $\dim Y \geq 1$.
Since $\dim N^{1}(X)_{\Q}=\dim N^{1}(Y)_{\Q}+1$ and $\delta_{f}> \delta_{g}$, 
$ \delta_{f}$ is an integer lager than $1$ and the generalized eigenspace of $f^{*}|_{N^{1}(X)}$ associated with $ \delta_{f}$ is dimension one.
Combined with Perron-Frobenius type theorem, this implies there exists a nef integral divisor $D \not\equiv 0$ such that $f^{*}D \sim_{\Q} \delta_{f}D$. 
Since $ \delta_{f}> \delta_{g}$, $D$ is $\pi$-ample.
By \cite[Lemma 9.1]{mz}, the ray $\R_{\geq0}D$ in $N^{1}(X)_{\R}$ is an extremal ray in $\Nef(X)$ and $ \overline{\Eff}(X)$. 
Let $R\subset \overline{NE}(X)$ be the ray contracted by $\pi$.
Take a positive rational number $a \in \Q_{>0}$ such that $B:=D+aK_{X}$ satisfies $B\cdot R=0$.

(i) Suppose $B$ is pseudo-effective.
Since $D$ generates an extremal ray of $ \overline{\Eff}(X)$, we have $\R_{\geq0}D=\R_{\geq0}(-K_{X})$.
Thus we get $f^{*}(-K_{X}) \sim_{\Q} \delta_{f}(-K_{X})$.
If $ \kappa(-K_{X})>0$, then by Proposition \ref{canht-posiitaka}, KSC holds for $f$ and (\ref{case:kscholds}) holds.
Suppose $ \kappa(-K_{X})=0$.
Since $f^{*}(-K_{X}) \sim_{\Q} \delta_{f}(-K_{X})$, $-K_{X}$ generates an extremal ray of $ \overline{\Eff}(X)$,
and $\Pic(X)_{\Q} = N^{1}(X)_{\Q}$, any effective $\Q$-divisor $D$ that is $\Q$-linearly equivalent to $-K_{X}$ has irreducible support.
Thus (\ref{case:k=0}) holds.

(ii) Suppose $B$ is not pseudo-effective.
Take an sufficiently small effective ample $\Q$-divisor $E$ on $X$ so that $(1/a)B+E$ is not pseudo-effective.
Let $A:=E+(1/a)D$. Note that this $\Q$-divisor is ample and
$K_{X}+A=E+(1/a)B$ is not pseudo-effective.
Take an effective $\Q$-divisor $ \Delta$ on $X$ such that $ \Delta\sim_{\Q} A$ and $(X, \Delta)$ is klt.
By \cite{bchm}, we can run $(K_{X}+ \Delta)$-MMP and get:
\[
\xymatrix{
X \ar@{-->}[r]^{\varphi} & X' \ar[d]^{\pi'} \\
 & Y'
}
\]
where $\varphi$ is the composite of $(K_{X}+ \Delta)$-flips and $(K_{X}+ \Delta)$-divisorial contractions, and
$\pi'$ is a fiber type contraction of $(K_{X'}+ \varphi_{*}\Delta)$-negative extremal ray. 
If $\rho(X)>\rho(X')$, then (\ref{case:mmp}) holds.
Suppose $\rho(X)=\rho(X')$.
Then $\varphi$ is the composite of flips and in particular it is isomorphic in codimension one.
Take any $n>0$ such that the above diagram is $f^{n}$-equivariant.
Consider the following equivariant diagram of linear maps:
\[
\xymatrix{
(f^{n})^{*} \acts N^{1}(X)_{\R} & N^{1}(X')_{\R} \ar[l]_{\varphi^{*}} \racts (f^{n}|_{X'})^{*}\\
(g^{n})^{*} \acts N^{1}(Y)_{\R} \ar@<-2.5ex>[u]^{\pi^{*}} & N^{1}(Y')_{\R} \racts (f^{n}|_{Y'})^{*} \ar@<5.8ex>[u]_{{\pi'}^{*}} 
}
\]
where $\varphi^{*}$ is the strict transform and it is an isomorphism.
Let $C'$ be a curve on $X'$ that is contracted by $\pi'$.
Then
\[
((K_{X'}+(1/a)\varphi_{*}D)\cdot C')=((K_{X'}+\varphi_{*} \Delta)\cdot C')-(\varphi_{*}E\cdot C')<0.
\]
Note that $(\varphi_{*}E\cdot C')\geq0$ since $\varphi_{*}E$ is effective and $\pi'$ is fiber type contraction.
Therefore, $K_{X'}+(1/a)\varphi_{*}D \notin {\pi'}^{*}(N^{1}(Y')_{\R})$.
On the other hand, $\varphi^{*}(K_{X'}+(1/a)\varphi_{*}D)=K_{X}+(1/a)D=(1/a)B \in \pi^{*}(N^{1}(Y)_{\R})$.
This implies $\pi^{*}(N^{1}(Y)_{\R})$ and $\varphi^{*}({\pi'}^{*}N^{1}(Y')_{\R})$ are different codimension one 
linear subspace of $N^{1}(X)_{\R}$.
Therefore, the induced map $N^{1}(Y')_{\R} \longrightarrow N^{1}(X)_{\R}/\pi^{*}(N^{1}(Y)_{\R})$ is surjective.
Since $(f^{n})^{*}$ acts on $N^{1}(X)_{\R}/\pi^{*}(N^{1}(Y)_{\R})$ as multiplication by $ \delta_{f^{n}}$,
$(f^{n}|_{Y'})^{*}$ has eigenvalue $ \delta_{f^{n}}$.
This implies $ \delta_{f^{n}|_{Y'}} = \delta_{f^{n}}= \delta_{f^{n}|_{X'}}$.
Thus, (\ref{case:mmp}) holds.
\end{proof}

The following lemmas are taken from \cite{mz}.

\begin{lem}[{cf. \cite[Theorem 6.2]{mz}}]\label{-Ktotinv}
Let $X$ be a $\Q$-factorial normal projective variety with $\Pic(X)_{\Q}=N^{1}(X)_{\Q}$.
Suppose $\kappa(-K_{X})=0$ and $-K_{X} \sim_{\Q} D\geq 0$.
Then for any surjective endomorphism $f \colon X \longrightarrow X$, we have
$f^{-1}(\Supp(D))=\Supp(D)$.
\end{lem}
\begin{proof}
This follows from the ramification formula and Lemma \ref{totinv} below.
\end{proof}

\begin{lem}[{cf. \cite[Proposition 6.1]{mz}}]\label{totinv}
Let $X$ be a $\Q$-factorial normal projective variety with $\Pic(X)_{\Q}=N^{1}(X)_{\Q}$.
Let $f \colon X \longrightarrow X$ be a surjective endomorphism and $D$ an effective $\Q$-divisor on $X$ such that:
\begin{enumerate}
\item $\kappa(D)=0$;
\item\label{ass2} $f^{*}D \sim_{\Q} D+B$ for some effective $\Q$-divisor $B$ on $X$.
\end{enumerate}
Then $f^{-1}(\Supp D)= \Supp D$.
\end{lem}
\begin{proof}
Pushing forward the equation in (\ref{ass2}), we have $(\deg f) D \sim_{\Q} f_{*}D+f_{*}B$.
Since $ \kappa(D)=0$, we have $ \kappa(f_{*}D)=0$ and $\Supp D = \Supp f_{*}D \cup \Supp f_{*}B \supset \Supp f_{*}D $.
Applying $f_{*}$ again, we get
\[
(\deg f) f_{*}D \sim_{\Q} f^{2}_{*}D+f^{2}_{*}B.
\] 
This implies $ \kappa(f^{2}_{*}D)=0$ and $\Supp f_{*}D \supset \Supp f^{2}_{*}D$.
Repeating this, we get
\[
\Supp D \supset \Supp f_{*}D \supset \Supp f^{2}_{*}D \supset \cdots.
\]
There exists a non-negative integer $n$ such that $\Supp f^{n}_{*}D=\Supp f^{n+1}_{*}D$.
\begin{claim}
If $\Supp f^{n}_{*}D=\Supp f^{n+1}_{*}D$, then $\Supp f^{n-1}_{*}D = \Supp f^{n}_{*}D$.
\end{claim}
\begin{claimproof}
Since $\Supp f^{n}_{*}D=\Supp f^{n+1}_{*}D$, we have 
\[
\Supp f^{*}(f^{n}_{*}D) = \Supp f^{*}(f^{n+1}_{*}D) = \Supp f^{*}f_{*}(f^{n}_{*}D).
\]
By projection formula, we have $f_{*}f^{*}=\deg f \id$ on $N^{1}(X)_{\Q}=\Pic X_{\Q}$.
Since $N^{1}(X)_{\Q}$ is finite dimensional, we have $f^{*}f_{*}=\deg f \id $ on $N^{1}(X)_{\Q}=\Pic X_{\Q}$.
Thus, $f^{*}f_{*}(f^{n}_{*}D) \sim_{\Q} \deg f(f^{n}_{*}D)$.
Moreover, we have 
\[
(f^{n})^{*}f^{n}_{*}D\sim_{\Q} f^{n}_{*}(f^{n})^{*}D = \deg f^{n} D.
\]
This shows $ \kappa(f^{n}_{*}D)=0$.
Therefore, we get $\Supp f^{*}(f^{n}_{*}D)=\Supp f^{n}_{*}D$.
Since $\kappa(f^{n-1}_{*}D)=0$ and $ f^{*}(f^{n}_{*}D)\sim_{\Q} (\deg f) f^{n-1}_{*}D$, we have 
\[
\Supp f^{*}(f^{n}_{*}D) = \Supp f^{n-1}_{*}D
\] 
and this implies $\Supp f^{n-1}_{*}D = \Supp f^{n}_{*}D$.
\end{claimproof}

By this claim, we get $\Supp D=\Supp f_{*}D$.
By the same argument, we get $f^{-1}(\Supp D)=\Supp f^{*}D=\Supp f^{*}f_{*}D = \Supp D$.
\end{proof}

\begin{rmk}\label{ramtot}
In the setting in Lemma \ref{-Ktotinv},
if $\Supp(D)$ is irreducible and $f \colon X \longrightarrow X$ is int-amplified, then
we have $f^{*}D = qD$ for some integer $q>1$.
Moreover, $D$ is a prime divisor and $R_{f}=(q-1)D \sim_{\Q} (q-1)(-K_{X})$.
\end{rmk}

Now we prove the main theorem of this paper.

\begin{thm}
Let $X$ be a smooth projective rationally connected variety over $\QQ$ admitting an int-amplified endomorphism.
Then Kawaguchi-Silverman conjecture holds for all surjective endomorphisms on $X$.
\end{thm}
\begin{proof}
See \cite[Definition 4.3]{deb} for the definition of rationally connected.
Note that since $X$ is rationally connected, all varieties $X'$ that appear in a process of MMP starting from $X$ are also rationally connected 
and therefore $N^{1}(X')_{\Q}=\Pic X'_{\Q}$.

Let $f \colon X \longrightarrow X$ be a surjective endomorphism.
Fix an int-amplified endomorphism $ \Phi \colon X \longrightarrow X$.
Note that $K_{X}$ is not pseudo-effective since $X$ is smooth rationally connected.
By \cite{bchm}, we can run $K_{X}$-MMP and get:
\[
\xymatrix{
X = X_{1} \ar@{-->}[r] & \cdots \ar@{-->}[r]& X_{r} \ar[d]^{\pi} \\
 & & Y
}
\]
where $X_{1} \dashrightarrow X_{r}$ is the composite of flips and divisorial contractions and $\pi$ is a fiber type contraction.
(Note that MMP ends up with fiber type contraction because $K_{X}$ is not pseudo-effective.) 
Take $n>0$ so that the above diagram is $f^{n}$ and $\Phi^{n}$-equivariant.
Apply Proposition \ref{3alt} to $f^{n}|_{X_{r}} \acts X_{r} \longrightarrow Y \racts f^{n}|_{Y}$.
If (\ref{case:kscholds}) holds, then KSC for $f$ holds by Remark \ref{birinvksc}.

Suppose (\ref{case:k=0}) holds.
By Lemma \ref{-Ktotinv} and Remark \ref{ramtot}, $R_{\Phi^{n}|_{X_{r}}}$ is irreducible and $\pi$-ample.
In particular, $\pi(R_{\Phi^{n}|_{X_{r}}})=Y$ and we can apply Theorem \ref{covlem} and conclude that 
$\Phi^{n}|_{X_{r}}$ satisfies $(\ast)$.
By Lemma \ref{liftlem}, $\Phi^{n}$ also satisfies $(\ast)$.
This is a contradiction because $X$ is algebraically simply connected and does not dominate an abelian variety of positive dimension.

Suppose (\ref{case:mmp}) holds.
We have the following $(K_{X_{r}}+ \Delta)$-MMP for some effective $\Q$-divisor $ \Delta$ with $(X, \Delta)$ is klt:
\[
\xymatrix{
X_{r} \ar@{-->}[r] & X_{r}' \ar[d]^{\pi'} \\
 & Y'
}
\]
where $X_{r} \dashrightarrow X_{r}'$ is birational and $\pi'$ is a contraction of fiber type.
The whole diagram is again $f^{m}$ and $\Phi^{m}$-equivariant for sufficiently divisible $m$.
By Proposition \ref{3alt}, we have either $\rho(X_{r})>\rho(X_{r}')$ or $ \delta_{f^{m}|_{X_{r}'}}= \delta_{f^{m}|_{Y'}}$.
In the first case, run $K_{X_{r}'}$-MMP. It ends up with Q-abelian variety or fiber type contraction (Theorem \ref{equiMMP}).
In the second case, run $K_{Y'}$-MMP. It ends up with Q-abelian variety or fiber type contraction (Theorem \ref{equiMMP}). 
In both cases, Q-abelian case does not occur because endomorphisms on Q-abelian varieties satisfy $(\ast)$ and we can get a contradiction as above
(In this case, we need to know that the condition $(\ast)$ ascends along fiber type contraction, which is also proved in Lemma \ref{liftlem}).
Write our new fiber type contraction $X^{(1)} \longrightarrow Y^{(1)}$. Note that $\rho(X_{r})>\rho(X^{(1)})$ or $\dim X_{r}> \dim X^{(1)}$.
For sufficiently divisible $m'>0$, the whole diagram is equivariant under $f^{m'}$ and $\Phi^{m'}$.
Then by Remark \ref{iterksc}, \ref{birinvksc}, and \ref{baseksc},  KSC for $f$ follows from KSC for $f^{m'}|_{X^{(1)}}$. 
Apply Proposition \ref{3alt} to $f^{m'}|_{X^{(1)}}$ and argue in the same way.
Repeating this process, we get a sequence $X^{(0)}=X_{r}, X^{(1)}, X^{(2)}, \dots$.
Since the Picard number or the dimension strictly decreases, this process finally terminates.
\end{proof}


\begin{thebibliography}{99}

\bibitem{ambro} Ambro, F.,
    \textit{The moduli b-divisor of an lc-trivial fibration}, Compos. Math. 141
(2005), no. 2, 385--403.


\bibitem{bchm} Birkar, C., Cascini, P.,  Hacon, C. D., and McKernan, J.,
	\textit{Existence of minimal models for varieties of log general type},
	J. Amer. Math. Soc., 23(2):405--468, 2010.

\bibitem{bg} Bombieri, E., Gubler, W.,
	\textit{Heights in Diophantine geometry},
	Cambridge university press, 2007.



\bibitem{brou-gon} Broustet, A., Gongyo, Y.,
	\textit{Remarks on log Calabi-Yau structure of varieties admitting polarized endomorphisms},
	Taiwanese J. Math. 21 (2017), no. 3, 569--582.

\bibitem{brou-hor} Broustet, A.,  H\"oring, A.,
	\textit{Singularities of varieties admitting an endomorphism},
	Math. Ann. 360 (2014), no. 1-2, 439--456.
	
\bibitem{cmz} Cascini, P.,  Meng, S., Zhang, D-Q.,
	\textit{Polarized endomorphisms of normal projective threefolds in arbitrary characteristic},
	arXiv:1710.01903v2.	


		

\bibitem{dang} Dang, N-B.,
	\textit{Degrees of iterates of rational maps on normal projective varieties},
	arXiv:1701.07760.


\bibitem{deb} Debarre, O.,
	\textit{Higher-dimensional algebraic geometry},
	Universitext. Springer-Verlag, New York, 2001. xiv+233 pp. 


\bibitem{df} Diller, J., Favre, C.,
	\textit{Dynamics of bimeromorphic maps of surfaces},
	 Amer.\ J. Math.\ {\bf 123} no.\ 6, 1135--1169, (2001).


\bibitem{dn} Dinh, T-C., Nguy\^{e}n, V-A.,
	\textit{Comparison of dynamical degrees for semi-conjugate meromorphic maps},
	Comment. Math. Helv. \textbf{86} (2011), no. 4, 817--840.
	
\bibitem{ds} Dinh, T-C., Sibony, N.,
	\textit{Equidistribution problems in complex dynamics of higher dimension},
	Internat. J. Math. \textbf{28}, 1750057 (2007).



\bibitem{gong} Gongyo, Y.,
	\textit{Abundance theorem for numerically trivial log canonical divisors of semi-log canonical pairs},
	J. Algebraic Geom. 22 (2013), no. 3, 549--564. 



\bibitem{hs} Hindry, M., Silverman, J. H., 
	\textit{Diophantine geometry. An introduction},
	Graduate Text in Mathematics, no.\ 20,
	Springer-Verlag, New York, 2000.


\bibitem{jonwul} Jonsson, M., Wulcan, E.,
	\textit{Canonical heights for plane polynomial maps of small topological degree},
	Math. Res. Lett., 19(6):1207--1217, 2012.
	

\bibitem{ks1} Kawaguchi, S., Silverman, J.\ H.,
	\textit{Dynamical canonical heights for Jordan blocks, arithmetic degrees of orbits, and nef canonical heights on abelian varieties},
	Trans.\ Amer.\ Math.\ Soc. \textbf{368} (2016), 5009--5035.


\bibitem{ks3} Kawaguchi, S., Silverman, J.\ H.,
	\textit{On the dynamical and arithmetic degrees of rational self-maps of algebraic varieties},  
	J. Reine Angew. Math. \textbf{713} (2016), 21--48. 


\bibitem{kolmor} Koll\'ar, J., Mori, S.,
	\textit{Birational geometry of algebraic varieties}, 1998, Cambridge Univ. Press. 

\bibitem{Lan}
Lang, S., 
	\textit{Fundamentals of Diophantine Geometry}, 
	Springer-Verlag, New York, 1983. 



\bibitem{ls} Lesieutre, J.,  Satriano, M.,
	\textit{Canonical heights on hyper-K\"{a}hler varieties and the Kawaguchi-Silverman conjecture},
	preprint, 2018, arXiv:1802.07388.

\bibitem{lin} Lin, J-L.,
	\textit{On the arithmetic dynamics of monomial maps}, arXiv:1704.02661.

\bibitem{ma} Matsuzawa, Y.,
	\textit{On upper bounds of arithmetic degrees}, 
	to appear in Amer. J. Math.

\bibitem{ma2} Matsuzawa, Y.,
	\textit{Kawaguchi-Silverman conjecture for endomorphisms on several classes of varieties},
	preprint.

\bibitem{ms}  Matsuzawa, Y., Sano, K.,
	\textit{Arithmetic and dynamical degrees of self-morphisms of semi-abelian varieties},
	to appear in  Ergodic Theory Dynam.\ Systems.

\bibitem{mss} Matsuzawa, Y.,  Sano, K., Shibata, T., 
	\textit{Arithmetic degrees and dynamical degrees of endomorphisms on surfaces}, 
	Algebra Number Theory 12 (2018), no. 7, 1635--1657.


\bibitem{mayo} Matsuzawa, Y., Yoshikawa, S.,
	\textit{Int-amplified endomorphisms on normal projective surfaces},
	preprint.

\bibitem{meng} Meng, S.,
	\textit{Building blocks of amplified endomorphisms of normal projective varieties},
	Math. Z. (2019). https://doi.org/10.1007/s00209-019-02316-7
	
\bibitem{meng-zhang} Meng, S. and Zhang, D.-Q.,
	\textit{Building blocks of polarized endomorphisms of normal projective varieties},
	Adv. Math. 325 (2018), 243--273.
		

\bibitem{meng-zhang2} Meng, S. and Zhang, D.-Q.,
	\textit{Semi-group structure of all endomorphisms of a projective variety admitting a polarized endomorphism},
	arXiv:1806.05828.

\bibitem{mz} Meng, S. and Zhang, D.-Q.,
	\textit{Kawaguchi-Silverman conjecture for surjective endomorphisms},
	arXiv:1908.01605



\bibitem{sil} Silverman, J.\ H.,
	\textit{Dynamical degree, arithmetic entropy, and canonical heights for dominant rational self-maps of projective space},
	Ergodic Theory Dynam.\ Systems \textbf{34} (2014), no.\ 2, 647--678.

\bibitem{sil2} Silverman, J.\ H.,
	\textit{Arithmetic and dynamical degrees on abelian varieties},
	J. Th\'eor. Nombres Bprdeaux {\bf 29} (2017), no. 1, 151--167.



\bibitem{tru0} Truong, T. T.,
	\textit{(Relative) dynamical degrees of rational maps over an algebraic closed field},
	arXiv:1501.01523v1.
	

	
\end{thebibliography}
\end{document}